\documentclass[11pt]{amsart}
\usepackage[utf8]{inputenc}
\usepackage{amsfonts, amssymb, amsmath, amsthm, color, float,enumerate}
\usepackage{url}
\usepackage{mathrsfs}
\usepackage[unicode,psdextra]{hyperref}
\usepackage{pgfplots,tikz}
\pgfplotsset{compat=1.18}

\title[]{FABIO-GLADIS-PABLO}
\author{}
\date{}

\usepackage[a4paper, left=2cm, right=2cm, top=3cm, bottom=3cm]{geometry} 

\theoremstyle{plain}
   \newtheorem{teo}{Theorem}
   
   \newtheorem{lema}[teo]{Lemma}
   \newtheorem{propo}[teo]{Proposition}
   
\theoremstyle{definition}
   
\theoremstyle{remark}

 \newtheorem{afirmacion}{Claim}

\numberwithin{equation}{section}
\numberwithin{teo}{section}
\allowdisplaybreaks

\definecolor{aquamarine}{rgb}{0.5, 1.0, 0.83}
\definecolor{americanrose}{rgb}{1.0, 0.01, 0.24}
\definecolor{arsenic}{rgb}{0.23, 0.27, 0.29}
\definecolor{blizzardblue}{rgb}{0.67, 0.9, 0.93}
\definecolor{blush}{rgb}{0.87, 0.36, 0.51}
\definecolor{celestialblue}{rgb}{0.29, 0.59, 0.82}
\definecolor{chocolate(web)}{rgb}{0.82, 0.41, 0.12}
\definecolor{brightpink}{rgb}{1,0,0.5}
\definecolor{cadmiunred}{rgb}{0.89,0,0.13}

\hypersetup{
	colorlinks = true,
	linkcolor = cadmiunred,
	anchorcolor = blue,
	citecolor = brightpink,
	filecolor = blue,
	urlcolor = blue
}

\begin{document}

    \title[Sawyer estimates of mixed type]{Sawyer estimates of mixed type for operators associated to a critical radius function}

\author[F. Berra]{Fabio Berra}
\address{CONICET and Departamento de Matem\'{a}tica (FIQ-UNL),  Santa Fe, Argentina.}
\email{fberra@santafe-conicet.gov.ar}

\author[G. Pradolini]{Gladis Pradolini}
\address{CONICET and Departamento de Matem\'{a}tica (FIQ-UNL),  Santa Fe, Argentina.}
\email{gladis.pradolini@gmail.com}
%
\author[P. Quijano]{Pablo Quijano}
\address{Instituto de Matem\'atica Aplicada del Litoral, CONICET-UNL, and Facultad de Ingenier\'ia Qu\'imica, UNL. \indent Colectora Ruta Nac. N 168, Paraje El Pozo.
3000 Santa Fe, Argentina.}
\email{pquijano@santafe-conicet.gov.ar}

\thanks{The authors were supported by CONICET, ANPCyT and UNL}



\maketitle

\begin{abstract}
We prove mixed inequalities for the Hardy-Littlewood maximal function $M^{\rho,\sigma}$, where $\rho$ is a critical radius function and $\sigma\geq 0$. We also exhibit and prove an extension of Cruz-Uribe, Martell and Pérez extrapolation result in \cite{CruzUribe-Martell-Perez} to the setting of Muckenhoupt weights associated to a critical radius function $\rho$. This theorem allows us to give mixed inequalities for Schrödinger-Calderón-Zygmund operators, extending some previous estimates that we have already proved in \cite{BPQ}. Since we are dealing with unrelated weights, the proof involves a quite subtle argument related with the original ideas from Sawyer in \cite{Sawyer}.
\end{abstract}

\section{Introduction}

Perhaps the key idea behind the proof of mixed inequalities arose in \cite{M-W-76}, where Muckenhoupt and Wheeden proved weighted weak type inequalities involving $A_p$ weights. Concretely, they proved that the estimate
\[\left|\left\{x\in \mathbb{R}: \mathcal{T}f(x)w^{1/p}(x)>t\right\}\right|\leq \frac{C}{t^p}\int_{\mathbb{R}}|f|^pw\]
holds for $p\geq 1$ and $w\in A_p$, where the operator $\mathcal{T}$ is either the Hardy-Littlewood  maximal operator or the maximal Hilbert transform. The proof relies upon the so-called corona decomposition, which involves the construction of intervals on the real line with certain good properties, called ``principal intervals''.

Some years later, by adapting the corona technique, Sawyer proved that the inequality
\begin{equation}\label{eq: mixta de Sawyer}
uv\left(\left\{x\in\mathbb{R}: \frac{M(fv)(x)}{v(x)}>t\right\}\right)\leq \frac{C}{t}\int_{\mathbb{R}}|f|uv    
\end{equation}
holds for $u,v\in A_1$ and every positive $t$. A motivation to study this estimate was to obtain a different proof of the $L^p$ boundedness of the Hardy-Littlewood maximal function for $A_p$ weights, given originally by Muckenhoupt in \cite{Muck72}. Although the estimate above generalizes the well-know weak $(1,1)$ type for $M$, the operator involved in \eqref{eq: mixta de Sawyer} is a perturbation of $M$ by the $A_1$ weight $v$. This modification notably changes the operator, and different techniques combined with the one mentioned above are required for the purpose to get the mentioned inequality.  

Many extensions and improvements of this inequality have been made on the last decades. In \cite{CruzUribe-Martell-Perez} the authors extended \eqref{eq: mixta de Sawyer} to higher dimensions including Calderón-Zygmund operators (CZO). The key idea was to first prove the mixed inequality for the dyadic Hardy-Littlewood maximal function and then obtain the same for $M$ and CZO through an extrapolation result on the scale of Lorentz spaces. The latter required $u\in A_1$ and $v\in A_\infty$ to hold, condition that made the authors conjecture that the proved mixed inequalities might also be true in this case. This conjecture was solved a few years ago in \cite{L-O-P}.

On the other hand, mixed inequalities have also been established for commutators of CZO with BMO symbol in \cite{Berra-Carena-Pradolini(M)}. The condition on the weigths, which had also been previously considered by the authors in \cite{CruzUribe-Martell-Perez}, involved a pair $u,v$ such that $u\in A_1$ and $v\in A_\infty(u)$. It is not difficult to see that this condition implies that the product $uv$ belongs to $A_\infty$. Under this assumption, many classical techniques of Harmonic Analysis, like Calderón-Zygmund decomposition, can be adapted in order to achieve the desired estimate. 

Mixed estimates were also studied on the fractional setting. In \cite{Berra-Carena-Pradolini(J)} the authors proved this type of inequalities for the fractional maximal operator $M_\gamma$ and for the fractional integral operator $I_\gamma$, where $0<\gamma<n$. The conditions on the weights include both the unrelated weights as well as the corresponding fractional version of the related condition stated above. As an application, mixed estimates for commutators with Lipschitz symbol were given. 

For generalized maximal and fractional related operators, mixed estimates have been obtained in \cite{Berra}, \cite{Berra-Carena-Pradolini(MN)}, and \cite{B22Pot}.

Regarding operators involving a critial radius function $\rho$, mixed estimates were obtained in \cite{BPQ} by considering related weights adapted to this setting. The inequalities included a variant of the Hardy-Littlewood maximal operator and certain type of CZO in this context. As an application, mixed estimates for operators on the Schrödinger setting were also established. The advantage of considering related weights is that many classical tools, like Calderón-Zygmund decomposition, can be performed in the new setting in order to get the main estimates. 

As far as we know, there are no preceding results of the type described above in the case of unrelated weights like those considered in \cite{CruzUribe-Martell-Perez} or \cite{L-O-P} adapted to Schrödinger operators. The purpose of this article is precisely to settle these type of problems. As well as it was done on the classical setting, we prove an extrapolation result that allows us to reduce the problem to the study of an adequate dyadic maximal operator. This extrapolation theorem deals with weights belonging to the $A_p^\rho$ Muckenhoupt classes considered in this context. Let us introduce the basics for our development. 

We shall consider the space $\mathbb{R}^d$ equipped with a \emph{critical radius function} $\rho\colon \mathbb{R}^d\to (0,\infty)$ whose variation is controlled by the existence of two constants $C_0$ and $N_0\geq 1$ such that the inequality
	\begin{equation} \label{eq: constantesRho}
	C_0^{-1}\rho(x) \left(1+ \frac{|x-y|}{\rho(x)}\right)^{-N_0}
	\leq \rho(y)
	\leq C_0 \,\rho(x) \left(1+ \frac{|x-y|}{\rho(x)}\right)^{\tfrac{N_0}{N_0+1}}
	\end{equation}
	holds for every $x,y\in\mathbb{R}^d$.

The classes of Muckenhoupt weights associated to this critical radius function $\rho$ were introduced in \cite{BHS-classesofweights} (see Section~\ref{seccion: preliminares} for details). 
	
 The following extrapolation result will be useful to derive mixed inequalities for operators associated to a critical radius function, by means of the corresponding estimate for the Hardy-Littlewood maximal operator in this framework.
 
\begin{teo}\label{thm: extrapolacion}
    Let $\mathcal{F}$ be a family of pairs of functions satisfying that there exists $0<p_0<\infty$ such that the inequality
    \begin{equation}\label{eq: thm: extrapolacion - desigualdad de Coifman}
      \int_{\mathbb{R}^d}|f(x)|^{p_0}w(x)\,dx\leq C\int_{\mathbb{R}^d}|g(x)|^{p_0}w(x)\,dx  
    \end{equation}
    holds for every $w\in A_\infty^\rho$, for every pair $(f,g)\in \mathcal{F}$ such that the left-hand side is finite and with $C$ depending only on $[w]_{A_\infty^\rho}$. Then the inequality
    \[\left\|\frac{f}{v}\right\|_{L^{1,\infty}(uv)}\leq C \left\|\frac{g}{v}\right\|_{L^{1,\infty}(uv)}\]
    holds for every $u\in A_1^\rho$ and  $v\in A_\infty^\rho$.
\end{teo}

For the classical version of this theorem see \cite{CruzUribe-Martell-Perez}.

\bigskip

Our first main result deals with the Hardy-Littlewood maximal operator in the setting of critical radius functions given, for $\sigma\geq 0$, by
\begin{equation*}
M^{\rho,\sigma}f(x)=\sup_{Q(x_0,r_0)\ni x} \left(1+\frac{r_0}{\rho(x_0)}\right)^{-\sigma}\left(\frac{1}{|Q|}\int_Q |f(y)|\,dy\right),
\end{equation*}
where $Q(x_0,r_0)$ is a cube in $\mathbb{R}^d$ with center $x_0$ and radius $r_0$. We also consider unrelated weights.

\begin{teo}\label{thm: mixta para M}
	Let $u\in A_1^{\rho}$ and $v\in A_\infty^\rho$. Then there exists $\sigma\geq0$ and $C>0$ such that the inequality 
	\begin{equation*}
	uv\left(\left\{x\in\mathbb{R}^d: \frac{M^{\rho,\sigma}(fv)(x)}{v(x)}>t\right\}\right)\leq \frac{C}{t}\int_{\mathbb{R}^d} |f(x)|u(x)v(x)\,dx
	\end{equation*}
	holds for every positive $t$.
\end{teo}

This result extends the weighted weak $(1,1)$ type for $M^{\rho,\theta}$ (see Proposition 4.2 in~\cite{BCH3}).

The extrapolation result given by Theorem~\ref{thm: extrapolacion} will allow us to get mixed estimates for Schrödinger Calderón-Zygmund operators (SCZO). 

Given $0<\delta\leq 1$, we say that a linear operator $T$ is a  \emph{Schrödinger-Calderón-Zygmund operator (SCZO) of $(\infty,\delta)$} type if
\begin{enumerate}[\rm(I)]
	\item $T$ is bounded from $L^1$ into $L^{1,\infty}$;
	\item $T$ has an associated kernel $K\colon\mathbb{R}^d\times\mathbb{R}^d\rightarrow\mathbb{R}$, in the sense that
	\begin{equation*}
	Tf(x)=\int_{\mathbb{R}^d} K(x,y)f(y)\,dy,\,\,\,\,
	f\in L_c^{\infty} \,\,\text{and a.e.}\,\,x\notin \text{supp}f;
	\end{equation*}
	\item
	for each $N>0$ there exists a constant $C_N$ such that
	\begin{equation}\label{TamPuntual}
	|K(x,y)| \leq
	\frac{C_N}{|x-y|^{d}} \left(1+ \frac{|x-y|}{\rho(x)}\right)^{-N},\,\,\, x\neq y, 
	\end{equation}		
	and there exists $C>0$ such that
	\begin{equation}\label{suav-puntual}
	|K(x,y)-K(x,y_0)|
	\leq C \frac{|y-y_0|^{\delta}}{|x-y|^{d+\delta}},\,\,\,\text{when}\,\,
	|x-y|>2|y-y_0|.
	\end{equation}
\end{enumerate}

Since these type of operators satisfy an estimate in the spirit of \eqref{eq: thm: extrapolacion - desigualdad de Coifman}, we can obtain mixed inequalities for them by combining Theorem~\ref{thm: extrapolacion} with Theorem~\ref{thm: mixta para M}.

\begin{teo}\label{thm: mixta para SCZO}
Let $0<\delta\leq 1$, $u\in A_1^\rho$ and $v\in A_\infty^\rho$. If $T$ is a SCZO of $(\infty,\delta)$ type, then there exists a positive constant $C$ for which the inequality
\begin{equation*}
	uv\left(\left\{x\in\mathbb{R}^d: \frac{|T(fv)(x)|}{v(x)}>t\right\}\right)\leq \frac{C}{t}\int_{\mathbb{R}^d} |f(x)|u(x)v(x)\,dx
	\end{equation*}
	holds for every positive $t$.
\end{teo}

The theorem above is an extension of the mixed inequality for SCZO established in \cite{BPQ}, where $u$ and $v$ are related weights in the sense that the product $uv$ belongs to $A_\infty^{\rho}$. Furthermore, if we take $v=1$ we recover the endpoint weak type of the corresponding operators (see Theorem 1 in~\cite{BHQ1} and Theorem 3.6 in~\cite{BCH3})

\section{Preliminaries and definitions}\label{seccion: preliminares}

We shall be dealing with cubes with sides parallel to the coordinate axes. When necessary, by $Q(x_Q,r_Q)$ we shall denote the cube centered at $x_Q$ and radius $r_Q=\sqrt{d}\ell(Q)/2$, where $\ell(Q)$ stands for the  length of the sides of $Q$.

We shall very often refer to \emph{critical cubes}, meaning cubes of the type $Q(x_0,\rho(x_0))$, and we call \emph{subcritical cubes} to those $Q(x_0,r)$ with $r\leq \rho(x_0)$. The family of all subcritical cubes will be denoted by $\mathcal{Q}_\rho$.

Let us introduce the classes of weights involved in our estimates. These types of Muckenhoupt $A_p$ classes related to $\rho$ were first defined in~\cite{BHS-classesofweights}.

Let $1<p<\infty$ and $\theta\geq 0$. We say that $w\in A_p^{\rho,\theta}$ if there exists a positive constant $C$ such that the inequality
\begin{equation}\label{eq: clase Ap,rho,theta}
\left(\frac{1}{|Q|}\int_Q w\right)^{1/p}\left(\frac{1}{|Q|}\int_Q w^{1-p'}\right)^{1/p'}\leq C\left(1+\frac{r}{\rho(x)}\right)^{\theta}
\end{equation} 
holds for every cube $Q=Q(x,r)$.

Similarly, $w\in A_1^{\rho,\theta}$ if
\begin{equation}\label{eq: clase A1,rho,theta}
\frac{1}{|Q|}\int_Q w\leq C\left(1+\frac{r}{\rho(x)}\right)^{\theta}\inf_Q w,
\end{equation} 
for every cube $Q=Q(x,r)$.  

We say that $w\in A_\infty^{\rho,\theta}$ if there exists a positive constant $C$ for which the inequality
\begin{equation}\label{eq: clase Ainf,rho,theta}
\left(\frac{1}{|Q|}\int_Q w\right)\text{exp}\left(\frac{1}{|Q|}\int_Q \log w^{-1}\right)\leq C\left(1+\frac{r}{\rho(x)}\right)^{\theta}
\end{equation} 
holds for every cube $Q=Q(x,r)$.
The smallest constants for which the inequalities above hold will be denoted by $[w]_{A_p^{\rho,\theta}}$.

For $1\leq p\leq \infty$, the $A_p^\rho$ class is defined as the collection of all the $A_p^{\rho,\theta}$ classes for $\theta\geq 0$, that is
\[A_p^\rho=\bigcup_{\theta\geq 0} A_p^{\rho,\theta}.\]
It can also be proven that 
\[A_\infty^\rho =\bigcup_{p\geq 1} A_p^\rho.\]

  We shall be dealing with an equivalent characterization of the weights defined in \eqref{eq: clase Ainf,rho,theta}. It can be seen that $w\in A_\infty^{\rho,\theta}$ if there exist positive constants $C$ and $\varepsilon$ such that the inequality
\begin{equation}\label{eq: condicion Ainf,rho,theta con epsilon}
\frac{w(E)}{w(Q)}\leq C\left(1+\frac{r}{\rho(x)}\right)^\theta\left(\frac{|E|}{|Q|}\right)^\varepsilon
\end{equation}
holds for every cube $Q=Q(x,r)$ and every measurable subset $E$ of $Q$.

For $1\leq p\leq \infty$, we also say that $w\in A_p^{\rho,\text{loc}}$ if the corresponding inequality \eqref{eq: clase A1,rho,theta}, \eqref{eq: clase Ap,rho,theta} or \eqref{eq: clase Ainf,rho,theta} is satisfied by every cube $Q\in \mathcal{Q}_\rho$. It is immediate from this definition that $A_p^{\rho,\text{loc}}\subset A_p^{\rho}$.

Recall that every $A_p^\rho$ weight verifies a reverse Hölder type inequality. Concretely, given $\theta\geq 0$ and $1<s<\infty$, we say that a weight $w$ belongs to the \textit{reverse H\"{o}lder class} $\mathrm{RH}_s^{\rho,\theta}$ if there exists a positive constant $C$ such that for every cube $Q=Q(x,r)$  the inequality
\begin{equation}\label{eq: clase RHs,rho,theta}
\left(\frac{1}{|Q|}\int_Q w^s\right)^{1/s}\leq C\left(1+\frac{r}{\rho(x)}\right)^\theta\left(\frac{1}{|Q|}\int_Q w\right).
\end{equation}
holds. When $s=\infty$, we say that $w\in \mathrm{RH}_{\infty}^{\rho,\theta}$ if
\begin{equation}\label{eq: clase RHinf,rho,theta}
\sup_Q w\leq C\left(1+\frac{r}{\rho(x)}\right)^\theta\left(\frac{1}{|Q|}\int_Q w\right)
\end{equation}
holds for every cube $Q=Q(x,r)$ and $C$ independent of $Q$. The smallest constant $C$ for which these estimates hold will be denoted by $[w]_{\mathrm{RH}_s^{\rho,\theta}}$.

As in the case of $A_p^\rho$ classes, we define
\[\mathrm{RH}_s^\rho=\bigcup_{\theta\geq 0}\mathrm{RH}_s^{\rho,\theta}, \quad 1<s\leq\infty.\]
Muckenhoupt weights associated to a critical radius function are closely related with continuity properties of the Hardy-Littlewood maximal function in this general setting, which is defined by 
\begin{equation}\label{eq: operador maximal de H-L}
M^{\rho,\sigma}f(x)=\sup_{Q(x_0,r_0)\ni x} \left(1+\frac{r_0}{\rho(x_0)}\right)^{-\sigma}\left(\frac{1}{|Q|}\int_Q |f(y)|\,dy\right),
\end{equation}
for a locally integrable function $f$ and $\sigma\geq 0$. Related with this operator, for $q\geq 1$ we define 
\begin{equation}\label{eq: operador maximal de H-L q}
M_q^{\rho,\sigma}f(x)=\left(M^{\rho,\sigma}(|f|^q)(x)\right)^{1/q}.
\end{equation}

It will be especially interesting to consider a localized version of the operator $M^{\rho,\sigma}$. Given a cube $R$, the \emph{localized Hardy-Littlewood maximal operator} is given by
\begin{equation}\label{eq: maximal clasica localizada}
  M_Rf(x)=\sup_{Q\ni x, Q\subseteq R}\frac{1}{|Q|}\int_Q |f|,  
\end{equation}
for $f\in L^1_{\text{loc}}$.
We shall further consider a dyadic version of the operator above. By a \emph{dyadic grid} $\mathcal{D}$ we understand a collection of cubes in $\mathbb{R}^d$ with the following properties:
\begin{enumerate}
	\item every cube  $Q$ in $\mathcal{D}$ verifies $\ell(Q)=2^k$, for some $k\in\mathbb{Z}$;
	\item if $P$ and $Q$ are in $\mathcal{D}$ and $P\cap Q\neq\emptyset$, then either $P\subseteq Q$ or $Q\subseteq P$;
	\item $\mathcal{D}_k=\{Q\in \mathcal{D}: \ell(Q)=2^k\}$ is a partition of $\mathbb{R}^d$, for every $k\in \mathbb{Z}$.
\end{enumerate}

By $\mathscr{D}(R)$ we understand the collection of dyadic subcubes of $R$.  A dyadic version of \eqref{eq: maximal clasica localizada} is given by
\begin{equation}\label{eq: maximal diadica localizada}
M_{R}^{\mathscr{D}}f(x)=\sup_{Q\ni x, Q\in \mathscr{D}(R)} \frac{1}{|Q|}\int_Q |f(y)|\,dy.
\end{equation}

\section{Proof of Theorem~\ref{thm: mixta para M}}

In order to give our mixed estimate for $M^{\rho,\sigma}$ we shall require the following result that can be found as Lemma 2.3 in~\cite{DZ-99}.

\begin{propo}\label{prop-cubrimientocritico}
	There exists a sequence of points $\{x_j\}_{j\in\mathbb{N}}$ such that the family of critical cubes given by $Q_j=~Q(x_j,\rho(x_j))$  satisfies
	\begin{enumerate}[\rm(a)]
		\item \label{item: prop-cubrimientocritico - item a}$\displaystyle \bigcup_{j\in\mathbb{N}} Q_j= \mathbb{R}^d$.
		\item \label{item: prop-cubrimientocritico - item b}There exist positive constants $C$ and $N_1$ such that for any $\sigma\geq1$,
		$\displaystyle \sum_{j\in\mathbb{N}}\mathcal{X}_{\sigma Q_j}\leq C \sigma^{N_1}$.
	\end{enumerate}
\end{propo}

The following lemma establishes a relation between the localized Hardy-Littlewood maximal operator and its dyadic version.  

\begin{lema}[\cite{BPQ}, Lemma 10]\label{lem-Mloc a Mdiad}
	There exists dyadic grids $\mathcal{D}^{(i)}$, $1\leq i\leq 3^d$ with the following property:
	For every cube $Q$ in $\mathbb{R}^d$ there exists $3^d$ dyadic cubes $Q_i\in \mathcal{D}^{(i)}$, $1\leq i\leq 3^d$, such that
	\[M_Q f(x)\leq 3^d \sum_{i=1}^{3^d}M_{Q_i}^{\mathscr{D}}(f\mathcal{X}_{Q})(x),\]
	for every $x\in Q$. Furthermore, each $Q_i$ verifies $Q\subseteq Q_i\subseteq \lambda Q$, where $\lambda$ depends only on $d$.
\end{lema}

The next lemma establishes a Calderón-Zygmund decomposition when we are restricted to a fixed cube. We give a proof for the sake of completeness.

\begin{lema}[Calderón-Zygmund decomposition on a cube]\label{lema: descomposicion de CZ en un cubo fijo}
Let $\lambda>0$, $f$ be a bounded function with compact support and $R$ be a cube in $\mathbb{R}^d$ such that $\displaystyle \frac{1}{|R|}\int_R |f|\leq \lambda$. Then there exists a family of maximal cubes $\Lambda=\{Q_j\}_j\subseteq \mathscr{D}(R)$ that verifies
\[\left\{x\in R: M_R^{\mathscr{D}}f(x)>\lambda\right\}=\bigcup_j Q_j\]
and $\displaystyle \lambda<\frac{1}{|Q_j|}\int_{Q_j}|f|\leq 2^d\lambda$, for every $j$.
\end{lema}

\begin{proof}
For $k\in \mathbb{N}_0$, let $\mathcal{D}_k=\{Q\in \mathscr{D}(R): \ell(Q)=2^{-k}\ell(R)\}$. Since
\[\frac{1}{|R|}\int_R |f|\leq\lambda,\]
we divide the cube $R$ by bisecting each side of it, obtaining $2^d$ subcubes in $\mathcal{D}_1(R)$. We pick those cubes $Q$ in $\mathcal{D}_1(R)$ such that 
\[\frac{1}{|Q|}\int_Q|f|>\lambda\]
and we bisect the others. We can repeat this process indefinitely to obtain that
\[\left\{x\in R: M_R^{\mathscr{D}}f(x)>\lambda\right\}=\bigcup_j Q_j.\]
In order to see that the cubes of the decomposition are maximal, if a cube $Q\in \mathcal{D}_k(R)$ was chosen, then its ancestor $Q'\in \mathcal{D}_{k-1}(R)$ was not, so
\[\lambda<\frac{1}{|Q|}\int_Q |f|\leq \frac{|Q'|}{|Q|}\frac{1}{|Q'|}\int_{Q'} |f|\leq 2^d\lambda.\qedhere\]
\end{proof}

\bigskip

We shall introduce some useful notation required for the proof of Theorem~\ref{thm: mixta para M}. As we shall see, it will be enough to show that
\begin{equation}\label{eq: estimacion en un dilatado de un cubo critico}
uv\left(\left\{x\in R:\frac{M^{\mathscr{D}}_R(fv)(x)}{v(x)}>t\right\}\right)\leq \frac{C}{t}\int_{R} |f|uv,
\end{equation}
where $R$ is a fixed cube obtained from dilating a critical cube. In order to simplify we shall assume  
that $t=1$ and $g=|f|v$ is a bounded function with compact support. Fixed a number $a>2^n$, we proceed as follows
\begin{align*}
  uv\left(\left\{x\in R:\frac{M^{\mathscr{D}}_R g(x)}{v(x)}>1\right\}\right)&=\sum_{k\in \mathbb{Z}} uv\left(\left\{x\in R:\frac{M^{\mathscr{D}}_R g(x)}{v(x)}>1, a^k<v\leq a^{k+1}\right\}\right)\\
  &=\sum_{k\in\mathbb{Z}}uv(E_k).
\end{align*}
Let $\Omega_k=\left\{x\in R: M_R^{\mathscr{D}}g(x)>a^k\right\}$. There exists a unique $k_0\in\mathbb{Z}$ such that
\[a^{k_0-1}< \frac{1}{|R|}\int_R g\leq a^{k_0}.\]
Thus we obtain that $\Omega_k=R$ for every $k\leq k_0-1$. Since $R$ is a fixed dilation of a critical cube, this allows us to write
\begin{align*}
    \sum_{k\leq k_0-1}uv(E_k)&\leq \sum_{k\leq k_0-1}a^{k+1-k_0}u(R)a^{k_0}\\
    &\leq \sum_{k\leq k_0-1} a^{k-k_0}u(R)\frac{a^2}{|R|}\int_R g\\
    &\leq C_{\theta} a^2[u]_{A_1^{\rho,\theta}}\left(\int_R|f|uv\right)\sum_{k\leq k_0-1}a^{k-k_0}\\
    &=C_{\theta}a^2[u]_{A_1^{\rho,\theta}}\left(\int_R|f|uv\right)\sum_{j=1}^\infty a^{-j}\\
    &=C_{\theta}\frac{a^2}{a-1}[u]_{A_1^{\rho,\theta}}\int_R|f|uv.
\end{align*}
Therefore, it will be enough to estimate the sum for $k\geq k_0$. Recall that for these index we have $\displaystyle \frac{1}{|R|}\int_R g\leq a^{k}$.

By Lemma~\ref{lema: descomposicion de CZ en un cubo fijo}, for every $k\geq k_0$ there exists a family of dyadic subcubes $\{Q_j^k\}_j\subseteq \mathscr{D}(R)$ such that $\Omega_k=\bigcup_j Q_j^k$ and
\begin{equation}\label{eq: promedio de g como a^k}
  a^k<\frac{1}{|Q_j^k|}\int_{Q_j^k}g\leq 2^da^k,\quad\text{ for every }j.  
\end{equation}
We proceed to classify these cubes in the following way. For $\ell\in \mathbb{N}_0$, we define
\[\Lambda_{\ell,k}=\left\{Q_j^k: a^{k+\ell}\leq\frac{1}{|Q_j^k|}\int_{Q_j^k}v< a^{k+\ell+1}\right\}\]
and
\[\Lambda_{-1,k}=\left\{Q_j^k: \frac{1}{|Q_j^k|}\int_{Q_j^k}v<a^{k}\right\}.\]
We further split the cubes on this last family. For each $Q_j^k\in \Lambda_{-1,k}$, we perform the Calderón-Zygmund decomposition of $v\mathcal{X}_{Q_j^k}$ at level $a^k$ in order to obtain a collection of subcubes $\{Q_{j,i}^k\}_i\subseteq \mathscr{D}(Q_j^k)$ such that
\begin{equation}\label{eq: promedio de v en cubos -1 como a^k}
    a^k<\frac{1}{|Q_{j,i}^k|}\int_{Q_{j,i}^k}v\leq 2^da^k,\quad\text{ for every }i.
\end{equation}
We also define, for $\ell\in \mathbb{N}_0$, the sets
\[\Gamma_{\ell,k}=\left\{Q_j^k\in \Lambda_{\ell,k}: \left|Q_j^k\cap \{x: a^k<v(x)\leq a^{k+1}\}\right|>0\right\}\]
and 
\[\Gamma_{-1,k}=\left\{Q_{j,i}^k: Q_j^k\in \Lambda_{-1,k} \text{ and } \left|Q_{j,i}^k\cap \{x: a^k<v(x)\leq a^{k+1}\}\right|>0\right\}.\]

The following two lemmas were established in \cite{L-O-P}. 

\begin{lema}\label{lema: control exponencial de u de un subconjunto del cubo}
If $\ell\geq 0$ and $Q_j^k\in \Gamma_{\ell,k}$, there exist two positive constants $c_1$ and $c_2$, depending only on $u$ and $v$, such that
\[u(E_k\cap Q_j^k)\leq c_1e^{-c_2\ell}u(Q_j^k)\]
\end{lema}

\begin{lema}\label{lema: sparsity de los cubos}
There exists a positive constant $C$, depending only on $a$ and $d$, such that
\[\left|\bigcup_{Q'\in \Gamma, Q'\subseteq Q} Q'\right|\leq C|Q|,\]
for every cube $Q\in \Gamma=\bigcup_{\ell\geq -1}\bigcup_{k\geq k_0} \Gamma_{\ell,k}$.
\end{lema}

We now proceed with the proof of the mixed estimates for $M^{\rho,\sigma}$.

\begin{proof}[Proof of Theorem~\ref{thm: mixta para M}]
Since we can write
\[M^{\rho,\sigma}f(x)\leq \sup_{Q\in \mathcal{Q}_\rho, Q\ni x}\frac{1}{|Q|}\int_Q |f|+\sup_{Q\not\in\mathcal{Q}_\rho, Q\ni x}\left(\frac{\rho(x_Q)}{r_Q}\right)^\sigma\left(\frac{1}{|Q|}\int_Q |f|\right)=M_{\text{loc}}^\rho f(x)+M_{\text{glob}}^{\rho,\sigma}f(x),\]
by using Proposition~\ref{prop-cubrimientocritico} there exists a sequence of critical cubes $\{S_j\}_j$ such that
\begin{align*}
uv\left(\left\{x\in \mathbb{R}^d: \frac{M^{\rho,\sigma}(fv)(x)}{v(x)}>1\right\}\right)&\leq \sum_{j\in\mathbb{N}}uv\left(\left\{x\in S_j: \frac{M^{\rho,\sigma}(fv)(x)}{v(x)}>1\right\}\right)\\
&\leq \sum_{j\in\mathbb{N}}uv\left(\left\{x\in S_j: \frac{M_{\text{loc} }^{\rho}(fv)(x)}{v(x)}>\frac{1}{2}\right\}\right)\\
&\qquad +\sum_{j\in\mathbb{N}}uv\left(\left\{x\in S_j: \frac{M_{\text{glob}}^{\rho,\sigma}(fv)(x)}{v(x)}>\frac{1}{2}\right\}\right).
\end{align*}
Let us first estimate the term corresponding to $M^\rho_{\text{loc}}$. Fix a cube $S_j=Q(x_j, \rho(x_j))$. There exists a cube $P_j=Q(x_j,C_\rho\, \rho(x_j))$ containing $S_j$ such that $M_{\text{loc}}^\rho(fv)(x)\leq M_{P_j}(fv)(x)$, for $x\in S_j$  (see~\cite{BPQ}).
By Lemma~\ref{lem-Mloc a Mdiad} there exist $3^d$ cubes $R_{i,j}$ satisfying $P_j\subseteq R_{i,j}\subseteq \lambda P_j$ for every $1\leq i\leq 3^d$ such that 
\[M_{P_j}(fv)(x)\leq \sum_{i=1}^{3^d} M_{R_{i,j}}^{\mathscr{D}}(fv\mathcal{X}_{P_j})(x).\]
By applying \eqref{eq: estimacion en un dilatado de un cubo critico} with these cubes $R_{i,j}$ and $f\mathcal{X}_{P_j}$, we arrive to
\begin{align*}
			uv\left(\left\{x\in S_j:\frac{M^\rho_{\textup{loc}}(fv)(x)}{v(x)}>\frac{1}{2}\right\}\right) 
			& \leq  C
			\sum_{i=1}^{3^d}uv\left(\left\{x\in R_{i,j}:\frac{M^{\mathscr{D}}_{R_{i,j}}(fv\mathcal{X}_{P_j})(x)}{v(x)}>\frac{1}{3^d2}\right\}\right)\\
			&\leq C\int_{P_j}|f|uv.
\end{align*}
If we sum over $j$, by item~\eqref{item: prop-cubrimientocritico - item b} of Proposition~\ref{prop-cubrimientocritico} we get
\[\sum_{j\in \mathbb{N}} uv\left(\left\{x\in S_j:\frac{M^\rho_{\textup{loc}}(fv)(x)}{v(x)}>\frac{1}{2}\right\}\right) 
			 \leq  C\int_{\mathbb{R}^d}|f|uv.\]
Therefore we obtain the desired bound for $M_{\text{loc}}^\rho$. The estimate for $M_{\text{glob}}^{\rho,\sigma}$ is straightforward and does not depend on the properties of the weight $v$.  Observe that, by~\eqref{eq: constantesRho}, there exist two positive constants $C$ and $c$ depending only on $\rho$ such that for $x\in S_j$ and $Q\ni x$, $Q\notin \mathcal{Q}_\rho$, there exists $k\in\mathbb{N}$ such that $2^{k-1}S_j\subseteq Q\subseteq 2^k S_j$ and $
 \rho(x_Q)/r_Q\leq C 2^{-ck}$. Therefore
	\begin{equation*}
	\begin{split}
	M^{\rho,\sigma}_{\textup{glob}}(fv)(x) & \leq C \sup_{k\geq 1} \frac{2^{-c k\sigma}}{|2^k S_j|}\int_{2^kS_j}  |f| v
	\\ & \leq C \sum_{k\geq 1} \frac{2^{-k(c\sigma-\theta)}}{u(2^kS_j)} \int_{2^kS_j} |f| u v  
	\\ & \leq \frac{C}{u(S_j)} \sum_{k\geq 1} 2^{-k(c\sigma-\theta)} \int_{2^kS_j} |f| u v = \frac{A_j}{u(S_j)}.
	\end{split}
	\end{equation*}
	
	Now, applying part~\eqref{item: prop-cubrimientocritico - item b} of Proposition~\ref{prop-cubrimientocritico}, 
	
	\begin{equation*}
	\begin{split}
	\sum_{j\in\mathbb{N}} uv  \left(\left\{x\in S_j:\frac{M^{\rho,\sigma}_{\textup{glob}}(fv)(x)}{v(x)}>\frac{1}{2}\right\}\right)
	 \leq & \sum_{j\in\mathbb{N}} uv  \left(\left\{x\in S_j:\frac{A_j}{u(S_j)v(x)}>\frac{1}{2}\right\}\right)
	\\ & \leq C \sum_{j\in\mathbb{N}} \frac{A_j}{u(S_j)}\int_{S_j}  u(x)\,dx 
	\\ & \leq C \sum_{j\in\mathbb{N}} \sum_{k\in\mathbb{N}}
	2^{-k(c\sigma-\theta)}  \int_{2^{k}S_j} |f| u v
	\\ & \leq C \sum_{k\in\mathbb{N}} 2^{-k(c\sigma-\theta)} \int_{\mathbb{R}^d}
	\left(\sum_{j\in\mathbb{N}}\mathcal{X}_{2^{k}S_j}(y)\right)
	|f(y)|u(y)v(y)\,dy
	\\ &   \leq C \int_{\mathbb{R}^d} |f(y)| u(y)v(y),
	\end{split}
	\end{equation*}
	choosing $\sigma>(N_1+\theta + 1)/c$. This completes the proof of the mixed estimate for $M^{\rho,\sigma}$.
	
	We proceed with the proof of \eqref{eq: estimacion en un dilatado de un cubo critico}. Recall that, by means of the previous calculations, it will be enough to show that
\[\sum_{k\geq k_0} uv(E_k)\leq C\int_R|f|uv.\]
Since $E_k\subseteq \Omega_k$ for every $k$, we proceed as follows
\begin{align*}
    \sum_{k\geq k_0} uv(E_k)&= \sum_{k\geq k_0} uv(E_k\cap \Omega_k)\\
    &=\sum_{k\geq k_0}\sum_j uv(E_k\cap Q_j^k)\\
    &\leq \sum_{k\geq k_0}\sum_{\ell\geq 0} \sum_{Q_j^k\in \Gamma_{\ell,k}}a^{k+1}u(E_k\cap Q_j^k)+ \sum_{k\geq k_0} \sum_{i: Q_{j,i}^k\in \Gamma_{-1,k}}a^{k+1}u\left(Q_{j,i}^k\right)\\
    &= I+II.
\end{align*}
We shall prove that there exists a positive constant $C$ such that
\begin{equation}\label{eq: acotacion de I+II por lado derecho}
    I+II\leq C\int_R |f|uv.
\end{equation}
Let us start with the estimate of $I$. For $\ell\geq 0$, we set $\mathscr{S}_\ell=\bigcup_{k\geq k_0}\Gamma_{\ell,k}$. We build the sequence of \emph{principal cubes} in $\mathscr{S}_\ell$ as follows
\[P_0^\ell=\{Q: Q \text{ is maximal in } \mathscr{S}_\ell \text{ in the sense of inclusion}\}\]
and, for $m\geq 0$ we say that $Q_j^k$ belongs to $P_{m+1}^\ell$ if there exists a cube $Q_s^t\in P_m^\ell$ such that
\begin{equation}\label{eq: thm: mixta para M - cubos principales (1)}
    \frac{1}{|Q_j^k|}\int_{Q_j^k}u>\frac{2}{|Q_s^t|}\int_{Q_s^t}u
\end{equation}
and this cube is maximal in this sense, that is
\begin{equation}\label{eq: thm: mixta para M - cubos principales (2)}
    \frac{1}{|Q_{j'}^{k'}|}\int_{Q_{j'}^{k'}}u\leq\frac{2}{|Q_s^t|}\int_{Q_s^t}u,
\end{equation}
for every $Q_j^k\subsetneq Q_{j'}^{k'}\subsetneq Q_s^t$. Take $P^\ell=\bigcup_{m\geq 0}P_m^\ell$, the set of principal cubes in $\mathscr{S}_\ell$. By applying Lemma~\ref{lema: control exponencial de u de un subconjunto del cubo} together with the definition of $\Lambda_{\ell,k}$ we get
\begin{align*}
    I&=\sum_{k\geq k_0}\sum_{\ell\geq 0} \sum_{Q_j^k\in \Gamma_{\ell,k}}a^{k+1}u(E_k\cap Q_j^k)\\
    &\leq \sum_{k\geq k_0}\sum_{\ell\geq 0} \sum_{Q_j^k\in \Gamma_{\ell,k}} a^{k+1}c_1e^{-c_2\ell}u(Q_j^k)\\
    &\leq \sum_{\ell\geq 0}c_1e^{-c_2\ell}a^{1-\ell}\sum_{k\geq k_0}\sum_{Q_j^k\in \Gamma_{\ell,k}}\frac{v(Q_j^k)}{|Q_j^k|}u(Q_j^k).
\end{align*}
 In order to compute the double inner sum, let us define
 \[\mathscr{A}_{(s,t)}^\ell=\left\{Q_j^k\in \bigcup_{k\geq k_0}\Gamma_{\ell,k}: Q_j^k\subseteq Q_s^t \text{ and } Q_s^t \text{ is the smallest principal cube in $P^\ell$ that contains it}.\right\}\]
 From this definition, every $Q_j^k\in \mathscr{A}_{(s,t)}^\ell$ is not principal, unless $Q_j^k=Q_s^t$. Recall that $v\in A_\infty^\rho$. This implies that $v\in A_\infty (R)$, then there exist two positive constants $C$ and $\varepsilon$ such that
 \begin{equation}\label{eq: thm: mixta para M - condicion Ainfty de v en R}
   \frac{v(E)}{v(Q)}\leq C\left(\frac{|E|}{|Q|}\right)^\varepsilon,  
 \end{equation}
 for every cube $Q\subseteq R$ and every measurable subset $E$ of $Q$.
 By combining \eqref{eq: thm: mixta para M - cubos principales (2)} with \eqref{eq: thm: mixta para M - condicion Ainfty de v en R} and Lemma~\ref{lema: sparsity de los cubos}, for every $\ell\geq 0$, we get 
 \begin{align*}
     \sum_{k\geq k_0}\sum_{Q_j^k\in \Gamma_{\ell,k}}\frac{v(Q_j^k)}{|Q_j^k|}u(Q_j^k)&=\sum_{Q_s^t\in P^\ell}\sum_{(k,j): Q_j^k\in \mathscr{A}_{(s,t)}^\ell}\frac{u(Q_j^k)}{|Q_j^k|}v(Q_j^k)\\
     &\leq 2\sum_{Q_s^t\in P^\ell}\frac{u(Q_s^t)}{|Q_s^t|}\sum_{(k,j): Q_j^k\in \mathscr{A}_{(s,t)}^\ell}v(Q_j^k).
     \end{align*}

By virtue of Lemma~\ref{lema: sparsity de los cubos} and the definition of the family $\Lambda_{\ell,k}$, notice that
\[\sum_{j: Q_j^k\in \mathscr{A}_{(s,t)}^\ell}|Q_j^k|\leq \sum_{j: Q_j^k\in \mathscr{A}_{(s,t)}^\ell} a^{-(k+\ell)}v(Q_j^k)\leq  a^{-(k+\ell)}v(Q_s^t)\leq a^{1+t-k}|Q_s^t|.\]
Since $v\in A_\infty$, we get
\begin{align*}
\sum_{(k,j): Q_j^k\in \mathscr{A}_{(s,t)}^\ell}v(Q_j^k)\leq \sum_{k\geq t} \sum_{j: Q_j^k\in \mathscr{A}_{(s,t)}^\ell} v(Q_j^k)
&\leq \sum_{k\geq t} v\left(\bigcup_{j: Q_j^k\in \mathscr{A}_{(s,t)}^\ell} Q_j^k\right)\\
&\leq C\sum_{k\geq t} v(Q_s^t) \left(\frac{\left|\bigcup_{j: Q_j^k\in \mathscr{A}_{(s,t)}^\ell} Q_j^k\right|}{|Q_s^t|}\right)^{\varepsilon}\\
& \leq C v(Q_s^t)\sum_{k\geq t} a^{(t-k)\varepsilon},
\end{align*}
so we arrive to
\[\sum_{k\geq k_0}\sum_{Q_j^k\in \Gamma_{\ell,k}}\frac{v(Q_j^k)}{|Q_j^k|}u(Q_j^k)\leq C\sum_{Q_s^t\in P^{\ell}} \frac{u(Q_s^t)}{|Q_s^t|}v(Q_s^t).\]

 Then, by \eqref{eq: promedio de g como a^k}, we have that
 \begin{align*}
    I&\leq  C\sum_{\ell\geq 0}e^{-c_2\ell}a^{-\ell}\sum_{Q_s^t\in P^\ell} \frac{u(Q_s^t)}{|Q_s^t|}v(Q_s^t)\\
    &\leq  C\sum_{\ell\geq 0}e^{-c_2\ell}\sum_{Q_s^t\in P^\ell} a^t u(Q_s^t)\\
    &\leq C\sum_{\ell\geq 0}e^{-c_2\ell}\sum_{Q_s^t\in P^\ell} \frac{u(Q_s^t)}{|Q_s^t|}\int_{Q_s^t}g\\
    &\leq C\sum_{\ell\geq 0}e^{-c_2\ell} \int_R |f(x)|v(x)\left(\sum_{Q_s^t\in P^\ell}\frac{u(Q_s^t)}{|Q_s^t|}\mathcal{X}_{Q_s^t}(x)\right)\,dx\\
    &\leq C\int_R |f(x)|h_1(x)v(x)\,dx.
 \end{align*}
 
 \begin{afirmacion}\label{afirmacion: h1 menor o igual que C u}
 There exists a positive constant $C$, independent of $\ell$, such that $h_1(x)\leq Cu(x)$, for almost every $x$.
 \end{afirmacion}
 By using this claim, we immediately have the desired bound for $I$.
 
 We now proceed to estimate $II$. Fix $0<\delta<\varepsilon$, where $\varepsilon$ is the constant appearing in \eqref{eq: thm: mixta para M - condicion Ainfty de v en R}. We shall perform the selection of principal cubes in $\mathscr{S}_{-1}=\bigcup_{k\geq k_0}\Gamma_{-1,k}$. Proceeding similarly as before, we take
 \[P_0^{-1}=\{Q: Q \text{ is maximal in }\mathscr{S}_{-1} \text{ in the sense of inclusion}\}\]
 and recursively, for $m\geq 0$, we say that $Q_{j,i}^k\in P_{m+1}^{-1}$ if there exists a cube $Q_{s,l}^t\in P_m^{-1}$ such that
 \begin{equation}\label{eq: thm: mixta para M - cubos principales -1 (1)}
    \frac{1}{|Q_{j,i}^k|}\int_{Q_{j,i}^k}u>\frac{a^{(k-t)\delta}}{|Q_{s,l}^t|}\int_{Q_{s,l}^t}u
\end{equation}
being maximal in this sense, that is
\begin{equation}\label{eq: thm: mixta para M - cubos principales -1 (2)}
    \frac{1}{|Q_{j',i'}^{k'}|}\int_{Q_{j',i'}^{k'}}u\leq\frac{a^{(k-t)\delta}}{|Q_{s,l}^t|}\int_{Q_{s,l}^t}u,
\end{equation}
for every $Q_{j,i}^k\subsetneq Q_{j',i'}^{k'}\subsetneq Q_{s,l}^t$. Let $P^{-1}=\bigcup_{m\geq 0}P_m^{-1}$, the set of principal cubes in $\mathscr{S}_{-1}$. We also define for this case the sets
\[\mathscr{A}_{(s,l,t)}^{-1}=\left\{Q_{j,i}^k\in \bigcup_{k\geq k_0}\Gamma_{-1,k}: Q_{j,i}^k\subseteq Q_{s,l}^t \text{ and } Q_{s,l}^t \text{ is the smallest cube in }P^{-1} \text{ that contains it}\right\}.\]
Then, by using \eqref{eq: promedio de v en cubos -1 como a^k} and \eqref{eq: thm: mixta para M - cubos principales -1 (2)}, we estimate $II$ as follows
\begin{align*}
    II&=\sum_{k\geq k_0}\, \sum_{i: Q_{j,i}^k\in \Gamma_{-1,k}}a^{k+1}u(Q_{j,i}^k)\\
    &\leq a\sum_{Q_{s,l}^t\in P^{-1}}\, \sum_{i,j,k: Q_{j,i}^k\in \mathscr{A}_{(s,l,t)}^{-1}}\frac{v(Q_{j,i}^k)}{|Q_{j,i}^k|}u(Q_{j,i}^k)\\
    &\leq a\sum_{Q_{s,l}^t\in P^{-1}} \frac{u(Q_{s,l}^t)}{|Q_{s,l}^t|}\sum_{k\geq t}a^{(k-t)\delta}\sum_{j,i: Q_{j,i}^k\in \mathscr{A}_{(s,l,t)}^{-1}}v(Q_{j,i}^k).
\end{align*}
For $k\geq t$, by \eqref{eq: promedio de v en cubos -1 como a^k}, we have that
\[\sum_{j,i: Q_{j,i}^k\in \mathscr{A}_{(s,l,t)}^{-1}}|Q_{j,i}^k|<\sum_{j,i: Q_{j,i}^k\in \mathscr{A}_{(s,l,t)}^{-1}}a^{-k}v(Q_{j,i}^k)\leq a^{-k}v(Q_{s,l}^t)\leq 2^da^{t-k}|Q_{s,l}^t|.\]
Then by virtue of \eqref{eq: thm: mixta para M - condicion Ainfty de v en R} we get
\[\sum_{j,i: Q_{j,i}^k\in \mathscr{A}_{(s,l,t)}^{-1}}v(Q_{j,i}^k) \leq C v(Q_{s,l}^t)\left(\frac{\left|\bigcup_{j,i: Q_{j,i}^k\in \mathscr{A}_{(s,l,t)}^{-1}}Q_{j,i}^k\right|}{|Q_s^t|}\right)^{\varepsilon}
  \leq Ca^{(t-k)\varepsilon}v(Q_{s,l}^t).\]
  Therefore, by \eqref{eq: promedio de v en cubos -1 como a^k},  \eqref{eq: promedio de g como a^k} and the fact that $0<\delta<\varepsilon$
\begin{align*}
  II&\leq C\sum_{Q_{s,l}^t\in P^{-1}} v(Q_{s,l}^t)\frac{u(Q_{s,l}^t)}{|Q_{s,l}^t|}\sum_{k\geq t}a^{(t-k)(\varepsilon-\delta)}\\
  &\leq C\sum_{Q_{s,l}^t\in P^{-1}} v(Q_{s,l}^t)\frac{u(Q_{s,l}^t)}{|Q_{s,l}^t|}\\
  &\leq C\sum_{Q_{s,l}^t\in P^{-1}} a^t u(Q_{s,l}^t)\\
  &\leq C \sum_{Q_{s,l}^t\in P^{-1}} \frac{u(Q_{s,l}^t)}{|Q_s^t|}\int_{Q_s^t}g\\
  &\leq C\int_R|f(x)|v(x)\left(\sum_{Q_{s,l}^t\in P^{-1}} \frac{u(Q_{s,l}^t)}{|Q_s^t|}\mathcal{X}_{Q_s^t}(x)\right)\,dx\\
  &=C\int_R|f(x)|h_2(x)v(x)\,dx.
\end{align*}

\begin{afirmacion}\label{afirmacion: h2 menor o igual que C u}
There exists a positive constant $C$ such that $h_2(x)\leq Cu(x)$ for almost every $x$.
\end{afirmacion}
This claim yields the desired estimate for $II$. Then we are done with $M_{\text{loc}}^\rho$.
\end{proof}

\medskip

We now give the proof of the claims. 

\begin{proof}[Proof of Claim~\ref{afirmacion: h1 menor o igual que C u}]
Fix $\ell\geq 0$ and $x\in R$ such that $u(x)<\infty$. Let us consider a sequence of nested principal cubes in $P^\ell$ that contain $x$. Let $Q^{(0)}$ be the maximal cube in $P^\ell$ that contains $x$. If we have chosen $j$ maximal principal cubes $Q^{(j)}$, $j\geq 0$, then we denote with $Q^{(j+1)}$ the maximal principal cube contained in $Q^{(j)}$ and that contains $x$. This sequence can only have a finite number of terms, since otherwise we would have
\[\frac{1}{|Q^{(0)}|}\int_{Q^{(0)}}u\leq \frac{2^{-j}}{|Q^{(j)}|}\int_{Q^{(j)}}u\leq C\frac{[u]_{A_1^\rho}}{2^j}u(x)\]
which would also yield
\[\frac{2^j}{|Q^{(0)}|}\int_{Q^{(0)}}u\leq C[u]_{A_1^\rho}u(x),\]
for every $j$. By letting $j\to\infty$ we would arrive to a contradiction. Then $x$ belongs to a finite number $J=J(x)$ of these cubes, and then
\begin{align*}
    h_1(x)&=\sum_{Q_s^t\in P^\ell}\frac{u(Q_s^t)}{|Q_s^t|}\mathcal{X}_{Q_s^t}(x)\\
    &\leq \sum_{j=0}^J\frac{1}{|Q^{(j)}|}\int_{Q^{(j)}}u\\
    &\leq \sum_{j=0}^{J}\frac{2^{j-J}}{|Q^{(J)}|}\int_{Q^{(J)}}u\\
    &\leq 2^{\theta}[u]_{A_1^{\rho,\theta}} u(x)\sum_{j=0}^J 2^{j-J}\\
    &\leq 2^{\theta}[u]_{A_1^{\rho,\theta}} u(x)2^{-J}(2^{J+1}-1)\\
    &\leq 2^{1+\theta}[u]_{A_1^{\rho,\theta}} u(x).\qedhere
\end{align*}
\end{proof}

\medskip

\begin{proof}[Proof of Claim~\ref{afirmacion: h2 menor o igual que C u}]
Fix $x\in R$ and assume again that $u(x)<\infty$. For every level $t$ there exists at most one cube on the decomposition of $\Omega_t$, $Q_s^t$, that contains $x$. If this is the case, we directly denote this cube as $Q^t$. Let $G=\{t: x\in Q^t\}$ and recall that $t\geq k_0$, so the set $G$ is bounded from below. We can also assume that $G$ is nonempty, since we would get $h_2(x)=0$ otherwise, and the desired estimate is trivially true. Then $G$ has a minimum $t_0$. This allows us to define a sequence in $G$ recursively: if $t_m$ has been chosen, with $m\geq 0$, we pick $t_{m+1}$ as the smallest element in $G$ that is greater than $t_m$ and verifies
\begin{equation}\label{eq: desigualdad elementos en G (1)}
    \frac{1}{|Q^{t_{m+1}}|}\int_{Q^{t_{m+1}}}u>\frac{2}{|Q^{t_m}|}\int_{Q^{t_m}}u.
\end{equation}
Then, if $t\in G$ and $t_m\leq t<t_{m+1}$ we have
\begin{equation}\label{eq: desigualdad elementos en G (2)}
    \frac{1}{|Q^{t}|}\int_{Q^{t}}u\leq \frac{2}{|Q^{t_m}|}\int_{Q^{t_m}}u.
\end{equation}
This sequence has only a finite number of terms. This can be seen similarly as we did in the proof of Claim~\ref{afirmacion: h1 menor o igual que C u}. Thus we have $\{t_m\}=\{t_m\}_{m=0}^M$. By denoting $\mathscr{F}_m=\{t\in G: t_m\leq t<t_{m+1}\}$, by \eqref{eq: desigualdad elementos en G (2)} we can write
\[h_2(x)=\sum_{Q_{s,l}^t\in P^{-1}}\frac{u(Q_{s,l}^t)}{|Q_s^t|}\mathcal{X}_{Q_s^t}(x)\leq \sum_{m=0}^M \left(\frac{2}{|Q^{t_m}|}\int_{Q^{t_m}}u\right)\sum_{t\in\mathscr{F}_m}\,\sum_{s,l: Q_{s,l}^t\in P^{-1}}\frac{u(Q_{s,l}^t)}{u(Q^t)}.\]
If we prove that there exists a positive constant $C$ such that
\begin{equation}\label{eq: thm: mixta para M, suma doble de cociente de u acotado por C}
  \sum_{t\in\mathscr{F}_m}\,\sum_{s,l: Q_{s,l}^t\in P^{-1}}\frac{u(Q_{s,l}^t)}{u(Q^t)}\leq C  
\end{equation}
we are done. Indeed, in this case we have
\begin{align*}
    h_2(x)&\leq \sum_{m=0}^{M}\frac{2C}{|Q^{t_m}|}\int_{Q^{t_m}}u\\
    &\leq C\sum_{m=0}^{M}2^{m-M}\frac{1}{|Q^{t_M}|}\int_{Q^{t_M}}u\\
    &\leq C[u]_{A_1^{\rho}}u(x)2^{-M}\sum_{m=0}^M2^m\\
    &\leq C[u]_{A_1^{\rho}}u(x).
\end{align*}
Let us prove that \eqref{eq: thm: mixta para M, suma doble de cociente de u acotado por C} holds. We fix $0\leq m\leq M$. If $t=t_0$, observe that
\[\sum_{s,l: Q_{s,l}^t\in P^{-1}}\frac{u(Q_{s,l}^t)}{u(Q^t)}\leq 1.\]
Now assume that $t_m<t<t_{m+1}$. Then if $Q_{s,l}^t\cap Q_{j,i}^{t_m}\neq \emptyset$, we must have $Q_{s,l}^t\subsetneq Q_{j,i}^{t_m}$. Let $Q_{j',i'}^{t'}$ be the smallest principal cube that contains $Q_{j,i}^{t_m}$. By combining \eqref{eq: thm: mixta para M - cubos principales -1 (1)} and \eqref{eq: thm: mixta para M - cubos principales -1 (2)} we have
\[\frac{1}{|Q_{s,l}^t|}\int_{Q_{s,l}^t}u>\frac{a^{(t-t')\delta}}{|Q_{j',i'}^{t'}|}\int_{Q_{j',i'}^{t'}}u\]
and
\[\frac{1}{|Q_{j,i}^{t_m}|}\int_{Q_{j,i}^{t_m}}u\leq \frac{a^{(t_m-t')\delta}}{|Q_{j',i'}^{t'}|}\int_{Q_{j',i'}^{t'}}u\]

Fixed $y\in Q_{s,l}^t$, by combining the estimates above with \eqref{eq: desigualdad elementos en G (2)} we get
\begin{align*}
    [u]_{A_1^{\rho}}u(y)&\geq  \frac{1}{|Q_{s,l}^t|}\int_{Q_{s,l}^t}u>\frac{a^{(t-t_m)\delta}}{|Q_{j,i}^{t_m}|}\int_{Q_{j,i}^{t_m}}u\\
    &>a^{(t-t_m)\delta}\inf_{Q_{j,i}^{t_m}}u
    \geq a^{(t-t_m)\delta}\inf_{Q^{t_m}}u\\
    &\geq \frac{a^{(t-t_m)\delta}}{[u]_{A_1^\rho}}\frac{1}{|Q^{t_m}|}\int_{Q^{t_m}}u\\
    &\geq \frac{a^{(t-t_m)\delta}}{2[u]_{A_1^\rho}}\frac{1}{|Q^{t}|}\int_{Q^{t}}u,
\end{align*}
which yields
\[u(y)\geq \frac{a^{(t-t_m)\delta}}{2[u]_{A_1^\rho}^2}\frac{1}{|Q^{t}|}\int_{Q^{t}}u:=\lambda,\]
for almost every $y\in Q_{s,l}^t$.

Since $u\in A_1^\rho\subseteq A_\infty^\rho$, there exist positive constants $C$ and $\nu$ such that
\[\frac{u(E)}{u(Q)}\leq C\left(\frac{|E|}{|Q|}\right)^\nu,\]
for every cube $Q\subseteq R$ and every measurable subset $E$ of $Q$. We can continue as follows
\begin{align*}
    \sum_{s,l: Q_{s,l}^t\in P^{-1}}u(Q_{s,l}^t)&\leq \frac{u(\{y\in Q^t: u(y)>\lambda\})}{u(Q^t)}u(Q^t)\\
    &\leq C\left(\frac{|\{y\in Q^t: u(y)>\lambda\}|}{|Q^t|}\right)^\nu u(Q^t)\\
    &\leq Cu(Q^t)\left(\frac{1}{\lambda |Q^t|}\int_{Q^t}u\right)^\nu\\
    &= Cu(Q^t)a^{(t_m-t)\delta\nu}.
\end{align*}
Recall that we have assumed $t_m<t<t_{m+1}$. For the case $t=t_m$ we have $Q_{s,l}^t=Q_{j,i}^{t_m}$ and the estimate follows similarly. Finally, we get
\[\sum_{t\in \mathscr{F}_m}\sum_{s,l: Q_{s,l}^t\in P^{-1}}\frac{u(Q_{s,l}^t)}{u(Q^t)}\leq C\sum_{t\geq t_m} a^{(t_m-t)\delta\nu}\leq C,\]
which proves \eqref{eq: thm: mixta para M, suma doble de cociente de u acotado por C}. This completes the proof of the claim.
\end{proof}

\section{Proof of Theorem~\ref{thm: extrapolacion} and its consequences}

We start this section by establishing some basic properties of $A_p^\rho$ weights. They will be required in order to prove the extrapolation given in Theorem~\ref{thm: extrapolacion}. For the first one see, for example, \cite{BHQ-twoweighted} or \cite{BHS-classesofweights}. We include a proof of the second one. Its corresponding classical version was established in \cite{CruzUribe-Martell-Perez}.

\begin{lema}\label{lema: propiedades clase Ap,rho}
	Let $1<p<\infty$. Then
	\begin{enumerate}[\rm(a)]
		\item \label{item: lema: propiedades clase Ap,rho - item a}if $u\in A_p^\rho$, then $u^{1-p'}\in A_{p'}^\rho$;
		\item \label{item: lema: propiedades clase Ap,rho - item b}if $u$ and $v$ belong to $A_1^\rho$, then $uv^{1-p}\in A_p^\rho$;
		\item \label{item: lema: propiedades clase Ap,rho - item c}if $w\in A_p^\rho$, there exist weights $u$ and $v$ in $A_1^\rho$ that verify $w=uv^{1-p}$.
	\end{enumerate}
\end{lema}

\begin{lema}\label{lema: producto de un peso por potencia epsilon de otro en A_1}
Let $u\in A_1^\rho$ and $v\in A_p^\rho$, $1\leq p<\infty$. Then there exists $0<\varepsilon_0<1$ such that $uv^\varepsilon\in A_p^\rho$, for every $0<\varepsilon<\varepsilon_0$.
\end{lema}

\begin{proof}
Since $u\in A_1^\rho$, there exists a number $s_0>1$ such that $u\in \text{RH}_{s_0}^\rho$. Set $\varepsilon_0=1/s_0'$, fix $0<\varepsilon<\varepsilon_0$ and take $s=(1/\varepsilon)'$. Since $s<s_0$ we also have $u\in \text{RH}_s^\rho$. We shall first consider the case $p=1$. Given any cube $Q=Q(x,r)$ we write
\begin{align*}
    \frac{1}{|Q|}\int_Q uv^\varepsilon &\leq \left(\frac{1}{|Q|}\int_Q u^s\right)^{1/s}\left(\frac{1}{|Q|}\int_Q v^{\varepsilon s'}\right)^{1/s'}\\
    &\leq [u]_{\text{RH}_s^{\rho,\theta_1}}\left(1+\frac{r}{\rho(x)}\right)^{\theta_1}\frac{u(Q)}{|Q|}\left[[v]_{{A_1}^{\rho,\theta_2}}\left(\inf_Q v\right)\left(1+\frac{r}{\rho(x)}\right)^{\theta_2}\right]^{1/s'}\\
    &\leq [u]_{\text{RH}_s^{\rho,\theta_1}}[v]_{{A_1}^{\rho,\theta_2}}^{1/s'}[u]_{A_1^{\rho,\theta_3}}\left(\inf_Q{uv^{\varepsilon}}\right)\left(1+\frac{r}{\rho(x)}\right)^{\theta_1+\theta_2/s'+\theta_3},
\end{align*}
which implies that $uv^{\varepsilon}\in A_1^{\rho,\sigma}$, where $\sigma=\theta_1+\theta_2/s'+\theta_3$.

We now deal with the case $1<p<\infty$. Set $\varepsilon$ and $s$ as before. For $Q=Q(x,r)$ we may write
\begin{align*}
    \left(\frac{1}{|Q|}\int_Q uv^{\varepsilon}\right)^{1/p}\left(\frac{1}{|Q|}\int_Q (uv^{\varepsilon})^{1-p'}\right)^{1/p'}&\leq \left(\frac{1}{|Q|}\int_Q u^s\right)^{1/(ps)}\left(\frac{1}{|Q|}\int_Q v\right)^{1/(s'p)}\\
    &\qquad\times \left(\frac{1}{|Q|}\int_Q u^{s(1-p')}\right)^{1/(sp')}\left(\frac{1}{|Q|}\int_Q v^{1-p'}\right)^{1/(s'p')}\\
    &\leq [v]_{A_{p}^{\rho,\theta_1}}^{1/s'}[u]_{\text{RH}_s^{\rho,\theta_2}}^{1/p}\left(\frac{u(Q)}{|Q|}\right)^{1/p}\left(\frac{1}{|Q|}\int_Q u^{s(1-p')}\right)^{1/(sp')}\\
    &\qquad \times \left(1+\frac{r}{\rho(x)}\right)^{\theta_1/s'+\theta_2/p}\\
    &\leq [v]_{A_{p'}^{\rho,\theta_1}}^{1/s'}[u]_{\text{RH}_s^{\rho,\theta_2}}^{1/p}[u]_{A_1^{\rho,\theta_3}}^{1/p}\left(1+\frac{r}{\rho(x)}\right)^{\theta_1/s'+(\theta_2+\theta_3)/p},
\end{align*}
which yields $uv^{\varepsilon}\in A_p^\rho$.
\end{proof}

We are now in a position to give the proof of Theorem~\ref{thm: extrapolacion}.

\begin{proof}[Proof of Theorem~\ref{thm: extrapolacion}]
Since $u\in A_1^{\rho}$ and $v\in A_\infty^\rho$, there exist nonnegative numbers $\theta_1$ and $\theta_2$ such that $u\in A_1^{\rho,\theta_1}$ and $v\in A_\infty^{\rho,\theta_2}$. For $\sigma\geq \theta_1$ to be chosen later, we define the auxiliary operator
\[S_{\rho,\sigma}f(x)=\frac{M^{\rho,\sigma}(fu)(x)}{u(x)}.\]
Let us first prove that $S_{\rho,\sigma}$ is bounded on $L^{\infty}(uv)$. Indeed, since $d\mu(x)=u(x)v(x)\,dx$ is absolutely continuous with respect to the Lebesgue measure, it will be enough to show that $S_{\rho,\sigma}$ is bounded on $L^\infty$. Fix $x\in \mathbb{R}^d$ and $Q=Q(x,r)$. We have that
\begin{align*}
    \frac{1}{u(x)}\left(1+\frac{r}{\rho(x)}\right)^{-\sigma}\frac{1}{|Q|}\int_Q |f|u&\leq \frac{\|f\|_{\infty}}{u(x)}\left(1+\frac{r}{\rho(x)}\right)^{-\sigma}\frac{1}{|Q|}\int_Q u\\
    &\leq [u]_{A_1^{\rho,\theta_1}}\|f\|_\infty \left(1+\frac{r}{\rho(x)}\right)^{-\sigma+\theta_1}\frac{\inf_Q u}{u(x)}\\
    &\leq [u]_{A_1^{\rho,\theta_1}}\|f\|_\infty,
\end{align*}
since we have chosen $\sigma\geq \theta_1$. Thus we get that
$S_{\rho,\sigma}f(x)\leq [u]_{A_1^{\rho,\theta_1}}\|f\|_\infty$, for almost every $x$, by taking supremum over $Q$. Finally, we get the desired bound by taking supremum on $x$.

Let us now prove that there exists a number $1<p_0<\infty$ such that $S_{\rho,\sigma}$ is bounded on $L^{p_0}(uv)$. Observe that
\[\int_{\mathbb{R}^d}(S_{\rho,\sigma}f(x))^{p_0}u(x)v(x)\,dx=\int_{\mathbb{R}^d}(M^{\rho,\sigma}(fu)(x))^{p_0}u^{1-p_0}(x)v(x)\,dx.\]
We have that $u^{1-p_0}v$ is a weight in $A_{p_0}^\rho$. Indeed, since $v\in A_\infty^\rho$ there exists a number $t>1$ such that $v\in A_t^\rho$. By item~\eqref{item: lema: propiedades clase Ap,rho - item c} in Lemma~\ref{lema: propiedades clase Ap,rho} there exists weights $v_1$ and $v_2$ in $A_1^\rho$ such that $v=v_1v_2^{1-t}$. Therefore
\[u^{1-p_0}v=v_1u^{1-p_0}v_2^{1-t}=v_1\left(uv_2^{(1-t)/(1-p_0)}\right)^{1-p_0}.\]
If we can prove that the weight $uv_2^{(1-t)/(1-p_0)}$ is in $A_1^{\rho}$, we would have that $uv^{1-p_0}\in A_{p_0}^\rho$ by means of item~\eqref{item: lema: propiedades clase Ap,rho - item b} in Lemma~\ref{lema: propiedades clase Ap,rho}. By Lemma~\ref{lema: producto de un peso por potencia epsilon de otro en A_1}, there exists $\varepsilon_0$ such that $uv_2^{\varepsilon}\in A_1^{\rho}$, for every $0<\varepsilon<\varepsilon_0$. If we take $p_0=1+2(t-1)/\varepsilon$, then $(1-t)/(1-p_0)=\varepsilon/2$,
and we have that $uv_2^{(1-t)/(1-p_0)}\in A_1^\rho$. 

Since $uv^{1-p_0}\in A_{p_0}^\rho$, there exists $\theta_3\geq 0$ such that $M^{\rho,\theta_3}$ is bounded on $L^{p_0}(uv^{1-p_0})$. Therefore, if we take $\sigma\geq \max\{\theta_1,\theta_3\}$ we get
\begin{align*}
\int_{\mathbb{R}^d}(S_{\rho,\sigma}f(x))^{p_0}u(x)v(x)\,dx&=\int_{\mathbb{R}^d}(M^{\rho,\sigma}(fu)(x))^{p_0}u^{1-p_0}(x)v(x)\,dx\\
&\leq C_0\int_{\mathbb{R}^d}|f(x)|^{p_0}u(x)v(x)\,dx.
\end{align*}
By applying Proposition~\ref{propo: interpolacion en espacios de Lorentz con medida} we get that
\[\|S_{\rho,\sigma}f\|_{L^{q,1}(uv)}\leq 2^{1/q}\left(C_0\left(\frac{1}{p_0}-\frac{1}{q}\right)^{-1}+C_1\right)\|f\|_{L^{q,1}(\mu)},\]
for every $p_0<q<\infty$. Particularly, we have that
\begin{equation}\label{eq: thm: extrapolacion - eq1}
\|S_{\rho,\sigma}f\|_{L^{q,1}(uv)}\leq K_0\|f\|_{L^{q,1}(uv)}
\end{equation}
for every $2p_0\leq q<\infty$ and with $K_0=4p_0(C_0+C_1)$.

By following the Rubio de Francia algorithm we can define the auxiliary operator
\[\mathcal{R}_{\rho,\sigma}h(x)=\sum_{k=0}^\infty \frac{S_{\rho,\sigma}^kh(x)}{(2K_0)^k}.\]
Then it follows immediately that 
\begin{equation}\label{eq: thm: extrapolacion - eq2}
    h(x)\leq \mathcal{R}_{\rho,\sigma}h(x) \quad\text{ and }\quad S_{\rho,\sigma}(\mathcal{R}_{\rho,\sigma}h)(x)\leq 2K_0\mathcal{R}_{\rho,\sigma}h(x).
\end{equation}
From the last inequality we get that $(\mathcal{R}_{\rho,\sigma}h)u\in A_1^{\rho,\sigma}$ and $[(\mathcal{R}_{\rho,\sigma}h)u]_{A_1^{\rho,\sigma}}\leq 2K_0$. By Lemma~\ref{lema: producto de un peso por potencia epsilon de otro en A_1} there exists a positive constant $\tilde\varepsilon_0$ depending only on $[(\mathcal{R}_{\rho,\sigma}h)u]_{A_1^{\rho,\sigma}}$ (and therefore, only on $K_0$) such that $(\mathcal{R}_{\rho,\sigma}h)uv_1^{\varepsilon}\in A_1^{\rho}$, for every $0<\varepsilon<\tilde\varepsilon_0$. We fix $0<\varepsilon<\min\{\tilde\varepsilon_0, 1/(2p_0)\}$. If we take $r=(1/\varepsilon)'$, then $r'>2p_0$. By \eqref{eq: thm: extrapolacion - eq1} we obtain that $S_{\rho,\sigma}$ is bounded on $L^{r',1}(uv)$ with constant $K_0$. This yields
\begin{equation}\label{eq: thm: extrapolacion - eq3}
\|\mathcal{R}_{\rho,\sigma}h\|_{L^{r',1}(uv)}\leq 2\|h\|_{L^{r',1}(uv)}.
\end{equation}
Now observe that
\[(\mathcal{R}_{\rho,\sigma}h)uv^{1/r'}=\left[(\mathcal{R}_{\rho,\sigma}h)uv_1^\varepsilon\right]v_2^{(1-t)/r'}=\left[(\mathcal{R}_{\rho,\sigma}h)uv_1^\varepsilon\right]v_2^{1-\ell},\]
where $\ell=1/r+t/r'>1$. By item~\eqref{item: lema: propiedades clase Ap,rho - item b} in Lemma~\ref{lema: propiedades clase Ap,rho}, this weight belongs to $A_{\ell}^\rho$, and therefore to $A_\infty^{\rho}$.
Fix a pair $(f,g)\in \mathcal{F}$ such that the left-hand side of \eqref{eq: thm: extrapolacion - desigualdad de Coifman} is finite. By the duality between $L^{r',1}(uv)$ and $L^{r,\infty}(uv)$ we have
\[\left\|fv^{-1}\right\|_{L^{1,\infty}(uv)}^{1/r}=\left\|\left(fv^{-1}\right)^{1/r}\right\|_{L^{r,\infty}(uv)}=\sup \left|\int_{\mathbb{R}^d}f^{1/r}(x)h(x)u(x)v^{1/r'}(x)\,dx\right|,\]
where the supremum is taken over the functions $h\in L^{r',1}(uv)$ with $\|h\|_{L^{r',1}(uv)}=1$. We fix $h$ and apply the hypothesis with $p=1/r$ and $w=(\mathcal{R}_{\rho,\sigma}h)uv^{1/r'}\in A_\infty^{\rho}$. Let us check that the left-hand side of \eqref{eq: thm: extrapolacion - desigualdad de Coifman} is indeed finite. By applying Hölder inequality for Lorentz spaces and using \eqref{eq: thm: extrapolacion - eq3}, we have
\begin{align*}
\left|\int_{\mathbb{R}^d}f^{1/r}(x)\mathcal{R}_{\rho,\sigma}h(x)u(x)v^{1/r'}(x)\,dx\right|&\leq \left\|\left(fv^{-1}\right)^{1/r}\right\|_{L^{r,\infty}(uv)}\|\mathcal{R}_{\rho,\sigma}h\|_{L^{r',1}(uv)}\\
&\leq 2\left\|\left(fv^{-1}\right)^{1/r}\right\|_{L^{r,\infty}(uv)}\|h\|_{L^{r',1}(uv)}\\
&<\infty.
\end{align*}
Then by applying \eqref{eq: thm: extrapolacion - desigualdad de Coifman}, \eqref{eq: thm: extrapolacion - eq1} and \eqref{eq: thm: extrapolacion - eq3} we get
\begin{align*}
  \left|\int_{\mathbb{R}^d}f^{1/r}(x)h(x)u(x)v^{1/r'}(x)\,dx\right| 
  &\leq \int_{\mathbb{R}^d}f^{1/r}(x)\mathcal{R}_{\rho,\sigma}h(x)u(x)v^{1/r'}(x)\,dx\\
  &\leq C\int_{\mathbb{R}^d}g^{1/r}(x)\mathcal{R}_{\rho,\sigma}h(x)u(x)v^{1/r'}(x)\,dx\\
  &\leq  C\left\|\left(gv^{-1}\right)^{1/r}\right\|_{L^{r,\infty}(uv)}\|\mathcal{R}_{\rho,\sigma}h\|_{L^{r',1}(uv)}\\
  &\leq 2C\left\|\left(gv^{-1}\right)^{1/r}\right\|_{L^{r,\infty}(uv)}\|h\|_{L^{r',1}(uv)}\\
  &\leq 2C\left\|gv^{-1}\right\|_{L^{1,\infty}(uv)}^{1/r}.
\end{align*}
This completes the proof.
\end{proof}

We finish this section with the corresponding proof of Theorem~\ref{thm: mixta para SCZO}. As an auxiliary tool we shall require the following result, proved in \cite{BCH13}, which states a Coifman type estimate between SCZO and an adequate maximal operator. 

\begin{teo}\label{teo: tipo Coifman para T (BCH)}
Let $0<p<\infty$, $0<\delta\leq 1$ and $w\in A_\infty^{\rho,\text{loc}}$. If $T$ is a SCZO of $(\infty,\delta)$ type, then for every $\theta>0$ there exists a positive constant $C$ such that the inequality
\[\int_{\mathbb{R}^d} |Tf(x)|^pw(x)\,dx\leq C\int_{\mathbb{R}^d} (M^{\rho,\theta}f(x))^pw(x)\,dx\]
holds for $f\in L^1_{\text{loc}}$.
\end{teo}

\begin{proof}[Proof of Theorem~\ref{thm: mixta para SCZO}]
Since $A_\infty^\rho\subset A_\infty^{\rho,\text{loc}}$, by combining Theorem~\ref{thm: extrapolacion} with  Theorem~\ref{teo: tipo Coifman para T (BCH)} we get that
\[\left\|\frac{T(fv)}{v}\right\|_{L^{1,\infty}(uv)}\leq C\left\|\frac{M^{\rho,\theta}(fv)}{v}\right\|_{L^{1,\infty}(uv)},\]
for every $\theta>0$. On the other hand, by virtue of Theorem~\ref{thm: mixta para M} there exists $\sigma>0$ and $C>0$ such that
\[uv\left(\left\{x\in\mathbb{R}^d: \frac{M^{\rho,\sigma}(fv)(x)}{v(x)}>t\right\}\right)
\leq \frac{C}{t}\int_{\mathbb{R}^d}|f(x)|u(x)v(x)\,dx,\]
for every $t>0$. By taking $\theta=\sigma$, we have that
\begin{align*}
\sup_{t>0} t\,uv\left(\left\{x\in\mathbb{R}^d: \frac{|T(fv)(x)|}{v(x)}>t\right\}\right)&\leq C\,\sup_{t>0} t\,uv\left(\left\{x\in\mathbb{R}^d: \frac{M^{\rho,\sigma}(fv)(x)}{v(x)}>t\right\}\right)\\
&\leq C\int_{\mathbb{R}^d}|f(x)|u(x)v(x)\,dx,
\end{align*}
which yields the desired estimate.
\end{proof}

\section{Application to Schrödinger type singular integrals}\label{seccion: aplicaciones}

We finish the article by exhibiting mixed estimates for operators associated to the Schrödinger semigroup generated by $L= -\Delta + V$ on $\mathbb{R}^d$ with  $d\geq 3$, as a consequence of our main results for SCZO. The function $V$ will be called the potential and we shall assume that it is nonnegative, not identically zero and also that $V\in \mathrm{RH}_q$, with $q>d/2$.

In~\cite{shen}, Shen introduced the function 
\[\rho(x)= \sup\left\lbrace r>0: \frac{1}{r^{d-2}} \int_{B(x,r)}V(x)\,dx \leq 1 \right\rbrace.\]
and proved that $\rho$ is a critical radius function, under the above assumptions.

Our aim is to obtain mixed weak type inequalities for the Schrödinger Riesz transforms of first and second order  given respectively by $\mathcal{R}_1=\nabla L^{-1/2}$ and $\mathcal{R}_2=\nabla^2 L^{-1}$. We shall also consider the operators $L^{i\alpha}$ for $\alpha\in\mathbb{R}$, $T_\gamma= V^\gamma L^{-\gamma}$ for $0\leq 1$ and $S_\gamma = V^{\gamma}\nabla L^{-1/2-\gamma}$, for $0<\gamma\leq 1/2$.

The following proposition establishes some known properties for the above operators. 

\begin{propo}\label{prop-lista}
	Let $d\geq 3$, $q>d/2$ and $V\geq 0$ be a potential such that $V\in\mathrm{RH}_q$. 
	Then, for some $0<\delta\leq 1$, which may differ at each occurrence, we have:
	\begin{enumerate}[\rm(a)]
		\item\label{item-lista0}
		For $\alpha \in\mathbb{R}$ , $L^{i\alpha}$ is a SCZO of $(\infty,\delta)$ type.
		\item\label{item-lista1}
		$ \mathcal{R}_1$ is a SCZO of $(\infty,\delta)$ type provided $q\geq d$.
		
		\item\label{item-lista3}
		$\mathcal{R}_2$ is a SCZO of $(\infty,\delta)$ type, provided $V$ satisfies
    \begin{equation}\label{eq-condextraV}
        |V(x) - V(y)| \leq C \frac{|x-y|^{\delta}}{\rho^{2+\delta}(x)}, \quad \text{ for } |x-y|<\rho(x).
    \end{equation} 
		\item\label{item-lista4}
		For $0<\gamma\leq 1$, $T_\gamma$ is a SCZO of $(\infty,\delta)$  type, provided $V$ satisfies  \begin{equation}\label{eq-condextraV2}
        |V^\gamma(x) - V^\gamma(y)| \leq C \frac{|x-y|^{\delta}}{\rho^{2\gamma+\delta}(x)}, \quad \text{ for } |x-y|<\rho(x).
    \end{equation} 
		\item\label{item-lista5} For $0<\gamma\leq 1/2$,
		$S_\gamma$ is a SCZO of $(\infty,\delta)$, provided that $V$ satisfies~\eqref{eq-condextraV2}
	\end{enumerate}
\end{propo}

\begin{proof}
    For~\eqref{item-lista0} and~\eqref{item-lista1} we refer to Theorem 0.4 and Theorem 0.8 in~\cite{shen}. To see~\eqref{item-lista3},~\eqref{item-lista4} and~\eqref{item-lista5} we refer to Proposition 1, Proposition 2 and Proposition 3 in~\cite{BHQ4}.
\end{proof}

As a consequence of this proposition combined with Theorem~\ref{thm: mixta para SCZO}, we obtain the following result. 

\begin{teo}
	Let $d\geq 3$, $q>d/2$ and $V\geq 0$ be a potential such that $V\in\mathrm{RH}_q$. Let $u\in A_1^\rho$ and $v\in A_\infty^\rho$. Then the following inequalities hold for every positive $t$.
	\begin{enumerate}[\rm(a)]
	    \item If $\alpha \in\mathbb{R}$,
	\[uv\left(\left\{x\in\mathbb{R}^d: \frac{| L^{i\alpha}(fv)(x)|}{v(x)}>t\right\}\right)\leq \frac{C}{t}\int_{\mathbb{R}^d}|f(x)|u(x)v(x)dx.\]
	 \item If $q >d$,
	\[uv\left(\left\{x\in\mathbb{R}^d: \frac{| \mathcal{R}_1(fv)(x)|}{v(x)}>t\right\}\right)\leq \frac{C}{t}\int_{\mathbb{R}^d}|f(x)|u(x)v(x)dx.\]
\item If $V$ satisfies~\eqref{eq-condextraV} for some $0<\delta\leq 1$,
	\[uv\left(\left\{x\in\mathbb{R}^d: \frac{| \mathcal{R}_2(fv)(x)|}{v(x)}>t\right\}\right)\leq \frac{C}{t}\int_{\mathbb{R}^d}|f(x)|u(x)v(x)dx.\]
 \item If $0<\gamma\leq 1$ and $V$ satisfies~\eqref{eq-condextraV2} for some $0<\delta\leq 1$,
	\[uv\left(\left\{x\in\mathbb{R}^d: \frac{| T_\gamma(fv)(x)|}{v(x)}>t\right\}\right)\leq \frac{C}{t}\int_{\mathbb{R}^d}|f(x)|u(x)v(x)dx.\]
 \item If $0<\gamma\leq 1/2$ and $V$ satisfies~\eqref{eq-condextraV2} for some $0<\delta\leq 1$,
	\[uv\left(\left\{x\in\mathbb{R}^d: \frac{| S_\gamma(fv)(x)|}{v(x)}>t\right\}\right)\leq \frac{C}{t}\int_{\mathbb{R}^d}|f(x)|u(x)v(x)dx.\]
 
\end{enumerate}
\end{teo}

As expected, in all of the above theorems we generalize the weighted weak $(1,1)$ type proved in Theorem 7.2 in~\cite{BCH3} for $L^{i\alpha}$ and $\mathcal{R}_1$, and easily derived from Theorem 3.6 in~\cite{BCH3} and Proposition~\ref{prop-lista} above  for $\mathcal{R}_2$, $T_\gamma$ and $S_\gamma$.

\section{Appendix}

Let $\mu$ be a measure that will be fixed throughout this section. Given a $\mu$-measurable function $f$, the \emph{distribution function} of $f$ is defined by $\lambda_f(s)=\mu(\{x\in\mathbb{R}^d: |f(x)|>s\})$. The \emph{decreasing rearrangement} $f^*$ is given by $f^*(t)=\inf\{s\geq 0: \lambda_f(s)\leq t\}$.

Some basic properties of the decreasing rearrangement function and its relation with the distribution function are included in the following proposition. A proof can be found in \cite{grafakos}.

\begin{propo}\label{propo: propiedades reordenada no creciente}
Given $f,g$ $\mu-$measurable functions and $0\leq s,t,t_1,t_2<\infty$ we have:
\begin{enumerate}[\rm(a)]
    \item \label{item: propo: propiedades reordenada no creciente - item a} $\lambda_f(f^*(t))\leq t$; 
    \item \label{item: propo: propiedades reordenada no creciente - item b}$f^*(t)>s$ if and only if $t<\lambda_f(s)$;
    \item \label{item: propo: propiedades reordenada no creciente - item c}$(f+g)^*(t_1+t_2)\leq f^*(t_1)+g^*(t_2)$;
    \item \label{item: propo: propiedades reordenada no creciente - item d} if $|f|\leq |g|$ almost everywhere, then $f^*\leq g^*$ and $|f|^*=f^*$;
    \item \label{item: propo: propiedades reordenada no creciente - item e} $f^*(0)=\|f\|_{L^\infty(\mu)}$.
\end{enumerate}
\end{propo}

Given $0<p,q\leq \infty$, we say that $f$ belongs to the \emph{Lorentz space} $L^{p,q}(\mu)$ if $\|f\|_{L^{p,q}(\mu)}<\infty$, where
\[\|f\|_{L^{p,q}(\mu)}=\left\{\begin{array}{ccl}
    \displaystyle\left(\int_0^\infty \left(t^{1/p}f^*(t)\right)^q\,\frac{dt}{t}\right)^{1/q} & if & q<\infty;\\
    \sup_{t>0} t^{1/p}f^*(t) & if & q=\infty.
\end{array}
\right.\]

The following proposition is an auxiliary tool in order to prove Theorem~\ref{thm: extrapolacion}, generalizing a result proved in \cite{CruzUribe-Martell-Perez}. We include a proof for the sake of completeness. 

\begin{propo}\label{propo: interpolacion en espacios de Lorentz con medida}
Let $1<p_0<\infty$ and $T$ be a sublinear operator that verifies
\[\|Tf\|_{L^{p_0,\infty}(\mu)}\leq C_0\|f\|_{L^{p_0,1}(\mu)} \quad\text{ and }\quad \|Tf\|_{L^{\infty}(\mu)}\leq C_1\|f\|_{L^{\infty}(\mu)}.\]
Then for every $p_0<p<\infty$ we have that
\[\|Tf\|_{L^{p,1}(\mu)}\leq 2^{1/p}\left(C_0\left(\frac{1}{p_0}-\frac{1}{p}\right)^{-1}+C_1\right)\|f\|_{L^{p,1}(\mu)}.\]
\end{propo}

\begin{proof}
We fix $t>0$ and write $f=f_t+f^t=f\mathcal{X}_{\{x: |f(x)|\leq f^*(t)\}}+f\mathcal{X}_{\{x: |f(x)|> f^*(t)\}}$. By applying items~\eqref{item: propo: propiedades reordenada no creciente - item c} and~\eqref{item: propo: propiedades reordenada no creciente - item d} of Proposition~\ref{propo: propiedades reordenada no creciente} we have that
\begin{align*}
\|Tf\|_{L^{p,1}(\mu)}&=\int_0^\infty t^{1/p}(Tf)^*(t)\,\frac{dt}{t}\\
&\leq 2^{1/p}\left(\int_0^\infty t^{1/p}(Tf_t)^*(t)\,\frac{dt}{t}+\int_0^\infty t^{1/p}\left(Tf^t\right)^*(t)\,\frac{dt}{t}\right)\\
&=2^{1/p}(A+B).
\end{align*}
In order to estimate $A$, by item~\eqref{item: propo: propiedades reordenada no creciente - item e} in Proposition~\ref{propo: propiedades reordenada no creciente} we get
\[(Tf_t)^*(t)\leq (Tf_t)^*(0)=\|Tf_t\|_{L^{\infty}(\mu)}\leq C_1\|f_t\|_{L^{\infty}(\mu)}\leq C_1f^*(t).\]
This leads to
\[A\leq C_1\int_0^\infty t^{1/p}f^*(t)\,\frac{dt}{t}=C_1\|f\|_{L^{p,1}(\mu)}.\]
For $B$, we first notice that $|f^t(x)|\leq |f(x)|$. By item~\eqref{item: propo: propiedades reordenada no creciente - item d} in Proposition~\ref{propo: propiedades reordenada no creciente} this implies that $\left(f^t\right)^*(s)\leq f^*(s)$ for every $s$. Moreover, by item~\eqref{item: propo: propiedades reordenada no creciente - item a} in Proposition~\ref{propo: propiedades reordenada no creciente} we get
\[\lambda_{f^t}(0)=\lambda_f(f^*(t))\leq t\]
and item~\eqref{item: propo: propiedades reordenada no creciente - item b} in the same proposition yields $(f^t)^*(t)=0$. Therefore  $(f^t)^*(s)=0$ for every $s\geq t$. We also observe that
\begin{align*}
    t^{1/p}\left(Tf^t\right)^*(t)&=t^{1/p-1/p_0}t^{1/p_0}\left(Tf^t\right)^*(t)\\
    &\leq t^{1/p-1/p_0}\left(\sup_{t>0}t^{1/p_0}\left(Tf^t\right)^*(t)\right)\\
    &=t^{1/p-1/p_0}\left\|Tf^t\right\|_{L^{p_0,\infty}(\mu)}.
\end{align*}
By combining these estimates, we get
\begin{align*}
    B&\leq \int_0^\infty t^{1/p-1/p_0}\left\|Tf^t\right\|_{L^{p_0,\infty}(\mu)}\frac{dt}{t}\\
    &\leq C_0\int_0^\infty t^{1/p-1/p_0}\left\|f^t\right\|_{L^{p_0,1}(\mu)}\,\frac{dt}{t}\\
    &=C_0\int_0^\infty t^{1/p-1/p_0}\left(\int_0^t s^{1/p_0}f^*(s)\,\frac{ds}{s}\right)\,\frac{dt}{t}\\
    &\leq C_0 \int_0^\infty\left(\int_s^{\infty}t^{1/p-1/p_0}\,\frac{dt}{t}\right)s^{1/p_0}f^*(s)\,\frac{ds}{s}\\
    &=C_0\left(\frac{1}{p_0}-\frac{1}{p}\right)^{-1}\int_0^\infty s^{1/p}f^*(s)\,\frac{ds}{s}\\
    &=C_0\left(\frac{1}{p_0}-\frac{1}{p}\right)^{-1}\|f\|_{L^{p,1}(\mu)}.
\end{align*}

The estimates for $A$ and $B$ allow us to conclude the thesis.
\end{proof}

\section*{Declarations}

\subsection*{Ethical approval}
Not applicable.
\subsection*{Competing interests}
The authors have no relevant financial or non-financial interests to disclose.
\subsection*{Author's contributions}
 All authors whose names appear on the submission made substantial contributions to the conception and design of the work, drafted the work and revised it critically, approved this version to be published, and agree to be accountable for all aspects of the work in ensuring that questions related to the accuracy or integrity of any part of the work are appropriately investigated and resolved.
\subsection*{Funding}
The authors were supported by PICT 2019 (ANPCyT) and CAI+D 2020 (UNL)
\subsection*{Availability of data and materials}
Not applicable.

\def\cprime{$'$}
\providecommand{\bysame}{\leavevmode\hbox to3em{\hrulefill}\thinspace}
\providecommand{\MR}{\relax\ifhmode\unskip\space\fi MR }
\providecommand{\MRhref}[2]{%
  \href{http://www.ams.org/mathscinet-getitem?mr=#1}{#2}
}
\providecommand{\href}[2]{#2}

\end{document}